\documentclass[letterpaper, 10 pt, conference]{ieeeconf}  

\IEEEoverridecommandlockouts                              

\usepackage{epsfig} 
\usepackage{graphicx}
\usepackage{amssymb}
\usepackage{amsmath,mathtools}
\usepackage{mathrsfs}
\usepackage{epstopdf}
\usepackage{color}
\usepackage{hyperref}
\usepackage{braket}
\usepackage{algorithmicx}
\usepackage{xifthen}
\usepackage{needspace}
\usepackage{algpseudocode}
\allowdisplaybreaks

\newtheorem{theorem}{Theorem}
\newtheorem{remark}{Remark}

\newtheorem{definition}{Definition}

\newtheorem{assumption}{Assumption}
\newcommand\numberthis{\addtocounter{equation}{1}\tag{\theequation}}

\newcommand{\mb}{\mathbf}
\newcommand{\mc}{\mathcal}
\DeclareMathOperator*{\argmin}{argmin}

\newcounter{protocol}% 
\newenvironment{protocol}[1][]%
  {%
    \vspace{-2pt}
     \needspace{2\baselineskip}% 
    \noindent \rule{\linewidth}{1pt} \endgraf% 
    \refstepcounter{protocol}% 
    \centering \textsc{Protocol}~\theprotocol%
    \ifthenelse{\isempty{#1}}{}{:\ #1}% 
  }{% 
  \vspace{-6pt}
  \noindent \rule{\linewidth}{1pt}\vspace{2pt}%
  }
\newcounter{myalg}% 
\newenvironment{myalg}[1][]%
  {%
  \vspace{-2pt}
     \needspace{2\baselineskip}% 
    \noindent \rule{\linewidth}{1pt} \endgraf% 
    \refstepcounter{protocol}% 
    \centering \textsc{Algorithm}~\theprotocol%
    \ifthenelse{\isempty{#1}}{}{:\ #1}% 
  }{% 
  \vspace{-6pt}
  \noindent \rule{\linewidth}{1pt}\vspace{2pt}%
  }

\title{\LARGE \bf
Cloud-based MPC with Encrypted Data
}

\author{Andreea B. Alexandru \and Manfred Morari \and George J. Pappas\\
\thanks{The authors are with the Department of Electrical and Systems Engineering, 
	University of Pennsylvania, Philadelphia, PA 19104.
        {\tt\footnotesize \{aandreea, morari, pappasg\}@seas.upenn.edu}.}%
}

\begin{document}

\maketitle
\thispagestyle{plain}
\pagestyle{plain}

%%%%%%%%%%%%%%%%%%%%%%%%%%%%%%%%%%%%%%%%%%%%%%%%%%%%%%%%%%%%%%%%%%%%%%%%%%%%%%%%
\begin{abstract}
This paper explores the privacy of cloud outsourced Model Predictive Control (MPC) for a linear system with input constraints. In our cloud-based architecture, a client sends her private states to the cloud who performs the MPC computation and returns the control inputs. In order to guarantee that the cloud can perform this computation without obtaining \textit{anything} about the client's private data, we employ a partially homomorphic cryptosystem. We propose protocols for two cloud-MPC architectures motivated by the current developments in the Internet of Things: a \mbox{client-server} architecture and a \mbox{two-server} architecture. In the first case, a control input for the system is privately computed by the cloud server, with the assistance of the client. In the second case, the control input is privately computed by two independent, non-colluding servers, with no additional requirements from the client. We prove that the proposed protocols preserve the privacy of the client's data and of the resulting control input. Furthermore, we compute bounds on the errors introduced by encryption. We present numerical simulations for the two architectures and discuss the trade-off between communication, MPC performance and privacy.
\end{abstract}

%%%%%%%%%%%%%%%%%%%%%%%%%%%%%%%%%%%%%%%%%%%%%%%%%%%%%%%%%%%%%%%%%%%%%%%%%%%%%%%%
\section{Introduction}
The increase in the number of connected devices, as well as their reduction in size and resources, determined a growth in the utilization of cloud-based services, in which a centralized powerful server offers on demand storage, processing and delivery capabilities to users. 
With the development of communication efficient algorithms, outsourcing computations to the cloud becomes very convenient. 
However, issues regarding the privacy of the shared data arise, as the users have no control over the actions of the cloud, which can leak or abuse the data it receives. 

Model Predictive Control (MPC) is a powerful scheme that is successfully deployed in practice~\cite{Mayne14model} for systems of varying dimension and architecture, including cloud platforms. 
In competitive scenarios, such as energy generation in the power grid, domestic scenarios, such as heating control in smart houses, or time-sensitive scenarios, such as traffic control, the control scheme should come with privacy guarantees to protect from eavesdroppers or from an untrustworthy cloud. 
For instance, in smart houses, \mbox{client-server} setups can be envisioned, where a local trusted computer aggregates the measurements from the sensors, but does not store their model and specifications and depends on a server to compute the control input or reference. The server can also posses other information, such as the weather. 
In a heating application, the parameters of the system can be known by the server, i.e., the energy consumption model of the house, but the data measurements and how much the owner wants to consume should be private. 
In traffic control, the drivers are expected to share their locations, which should remain private, but are not expected to contribute to the computation. Hence, the locations are collected and processed only at a single server's level, e.g., in a \mbox{two-server} setup, which then sends the result back to the cars or to traffic lights. 

Although much effort has been dedicated in this direction, a universally secure scheme that is able to perform locally, at the cloud level, any given functionality on the users' data has not been developed yet~\cite{vanDijk10}. 
For a single user and functionalities that can be described by boolean functions, fully homomorphic encryption (FHE)~\cite{GentryPhD,Brakerski11} guarantees privacy, but at high storage and complexity requirements~\cite{Poppelmann15}. 
For multiple users, the concept of functional privacy is required, which can be attained by functional encryption~\cite{Boneh11functional}, which was developed only for limited functionalities. 
More tractable solutions that involve multiple interactions between the participating parties to ensure the confidentiality of the users' data have been proposed. 
In \mbox{client-server} computations, where the users have a trusted machine, called the client, that performs computations of smaller intensity than the server, 
we mention partially homomorphic encryption (PHE)~\cite{Paillier99} and differential privacy (DP)~\cite{Dwork08}. 
In \mbox{two-server} computations, in which the users share their data to two \mbox{non-colluding} servers, the following solutions are available: secret sharing~\cite{Shamir79,Algesheimer02}, garbled circuits~\cite{Yao82,Bellare12}, Goldreich-Micali-Wigderson protocol~\cite{Goldreich87}, PHE~\cite{Erkin12}.

\subsection{Contributions} 
In this paper, we discuss the implicit MPC computation for a linear system with input constraints, where we privately compute a control input, while maintaining the privacy of the state, using a cryptosystem that is partially homomorphic, i.e., supports additions of encrypted data. 
In the first case we consider, the control input is privately computed by a server, with the help of the client. In the second case, the computation is performed by two \mbox{non-colluding} servers. 
The convergence of the state trajectory to the reference is public knowledge, so it is crucial to not reveal anything else about the state and other sensitive quantitities. 
Therefore, we use a privacy model that stipulates that \textit{no computationally efficient algorithm run by the cloud can infer anything about the private data}, or, in other words, an adversary doesn't know more about the private data than a random guess. 
Although this model is very strict, it thoroughly characterizes the loss of information.

This work explores fundamental issues of privacy in control: the trade-off between computation, communication, performance and privacy. We present two main contributions: proposing two privacy-preserving protocols for MPC and evaluating the errors induced by the encryption. 

\subsection{Related work}
In control systems, ensuring the privacy of the measurements and control inputs from eavesdroppers and from the controller has been so far tackled with differential privacy, homomorphic encryption and transformation methods. 
Kalman filtering with DP was addressed in~\cite{leNy14differentially}, current trajectory hiding in~\cite{Koufogiannis17}, linear distributed control~\cite{Wang2017differential}, and distributed MPC in~\cite{Zellner17}. 
The idea of encrypted controllers was introduced in~\cite{Kogiso15} and~\cite{Farokhi17}, using PHE, and in~\cite{Kim16encrypting} using FHE. Kalman filtering with PHE was further explored in~\cite{Xu17secure}. 
Optimization problems with DP are addressed in~\cite{Nozari16,Han17} and PHE in~\cite{Shoukry16,Freris16,Alexandru17,Lu2018privacy}. 
Many works proposed privacy through transformation methods that use multiplicative masking. While the computational efficiency of such methods is desirable, their privacy cannot be rigorously quantified, as required by our privacy model, since the distribution of the masked values is not uniform~\cite{deHoogh12phd}.

Recent work in~\cite{Darup18} has tackled the problem of privately computing the input for a constrained linear system using explicit MPC, in a \mbox{client-server} setup. 
There, the client performs the computationally intensive trajectory localization and sends the result to the server, which then evaluates the corresponding affine control law on the encrypted state using PHE. 
Although explicit MPC has the advantage of computing the parametric control laws offline, the evaluation of the search tree at the cloud's level is intractable when the number of nodes is large, since all nodes have to be evaluated in order to not reveal the polyhedra the in which the state lies, and comparison cannot be performed locally on encrypted data. 
Furthermore, the binary search in explicit MPC is intensive and requires the client to store all the characterization of the polyhedra, which we would like to avoid. Taking this into consideration, we focus on implicit MPC.

The performance degradation of a linear controller due to encryption is analyzed in~\cite{Kogiso2018upper}. In our work, we investigate performance degradation for the nonlinear control obtained from MPC. 

%%%%%%%%%%%%%%%%%%%%%%%%%%%%%%%%%%%%%%%%%%%%%%%%%%%%%%%%%%%%%%%%%%%%%%%%%%%%%%%%

\section{Problem setup}
We consider a discrete-time linear time-invariant system:
\begin{equation}\label{eq:system}
	x(t+1) = Ax(t) + Bu(t),
\end{equation}
with the state $x\in\mc X\subseteq\mathbb R^{n}$ and the control input $u\in\mc U\subseteq\mathbb R^{m}$. The optimal control receding horizon problem with constraints on the states and inputs can be written as:
\vspace{-\topsep}
\begin{align*}
	J_N^\ast(x(t)) =& \min\limits_{u_{0,\ldots,N-1}}\frac{1}{2} \left(x_N^\intercal P x_N + \sum_{k=0}^{N-1} x_k^\intercal Q x_k + u_k^\intercal R u_k \right)\\
	s.t.&~ x_{k+1} = Ax_k + Bu_k,~k=0,\ldots,N-1		\numberthis \label{eq:mpc}\\
	&~x_k\in \mc X,~u_k\in \mc U,~k=0,\ldots,N-1\\
	&~x_N \in \mc X_f,~~x_0 = x(t),
\end{align*}
where $N$ is the length of the horizon and $P,Q,R\succ 0$ are cost matrices. 
For reasons related to error bounding, explained in Section~\ref{sec:fixed-pointMPC}, in this paper, we consider input constrained systems: \mbox{$\mc X = \mathbb R^n$}, $0\in\mc U = \{-l_u\preceq u\preceq h_u\}$, and impose stability without a terminal state constraint, i.e. $ \mc X_f = \mathbb R^n$, but with appropriately chosen costs $P,Q,R$ and horizon $N$ such that the \mbox{closed-loop} system has robust performance to bounded errors due to encryption, which will be described in Section~\ref{sec:fixed-pointMPC}. 
A survey on the conditions for stability of MPC is given in~\cite{Mayne00}. 

Through straightforward manipulations,~\eqref{eq:mpc} can be written as a quadratic program (see details on obtaining the matrices $H$ and $F$ in~\cite[Ch.~8,11]{MPCbook17}) in the variable \mbox{$U:= \left[u_0 \ u_1 \ \ldots \ u_{N-1} \right]^\intercal$}. 
\begin{equation}\label{eq:mpc(i)}
	U^\ast(x(t))~= \argmin\limits_{U\in \mc U} \frac{1}{2}U^\intercal H U + U^\intercal F^\intercal x(t)
\end{equation}

For simplicity, we keep the same notation for the augmented constraint set $\mc U$. 
After obtaining the optimal solution, the first $m$ components of $U^\ast(x(t))$ are applied as input to the system~\eqref{eq:system}: \mbox{$u^\ast(x(t)) =(U^\ast(x(t)))_{1:m}$}. 
This problem easily extends to the case of following a reference.

\subsection{Solution without privacy requirements}
The constraint set $\mc U$ is a hyperbox, so the projection step required for solving~\eqref{eq:mpc(i)} has a simple closed form solution and the optimization problem can be efficiently solved with the projected Fast Gradient Method (FGM)~\cite{Nesterov13book}, given in Algorithm~\ref{alg:FGM}. 
The objective function is strongly convex, since \mbox{$H\succ 0$}, therefore we can use the constant step sizes \mbox{$L = \lambda_{max}(H)$} and \mbox{$\eta = (\sqrt{\kappa(H)}-1)/(\sqrt{\kappa(H)}+1)$}, where $\kappa(H)$ is the condition number of~$H$. 
Warm starting can be used at subsequent time steps of the receding horizon problem by using part of the previous solution~$U_K$ to construct a feasible initial iterate of the new optimization problem.

\begin{myalg}[Projected Fast Gradient Descent] \label{alg:FGM}
\begin{algorithmic}[1]
\small
  \Require{$H,F,x(t),\mc U,L,\kappa(H),\eta,U_0\in\mc U, z_0 = U_0,K$}
 \Ensure{$U_K(x(t))$}
\For{k=0\ldots,K-1}
 	\State $t_k = (\mathbf I_{Nm} - \frac{1}{L}H)z_k - \frac{1}{L}F^\intercal x(t)$
	\State $U_{k+1}^i = \begin{cases} -l_u^i, & \text{if }t_k^i < -l_u^i \\ t_k^i, & \text{if }t_k^i\in[-l_u^i,~h_u^i] \\ h_u^i, & \text{if }t_k^i > h_u^i\end{cases}$, $i=1,\ldots,Nm$
	\State $z_{k+1} = (1+\eta)U_{k+1} - \eta U_{k}$
 \EndFor
\end{algorithmic}
\end{myalg}

\subsection{Privacy objectives}
The unsecure cloud-MPC problem is depicted in Figure~\ref{fig:unencrypted_data}. 
The system's constant parameters $A,B,P,Q,R,N$ are public, motivated by the fact the parameters are intrinsic to the system and hardware, and could be guessed or identified; however, the measurements, control inputs and constraints are not known and should remain private. 
The goal of this work is to devise private cloud-outsourced versions of Algorithm~\ref{alg:FGM} such that the client obtains the control input $u^\ast(t)$ for system~\eqref{eq:system} with only a minimum amount of computation. 
The cloud (consisting of either one or two servers) should not infer anything else than what was known prior to the computation about the measurements $x(t)$, the control inputs $u^\ast(t)$ and the constraints~$\mc U$. 
We tolerate \textbf{\mbox{semi-honest}} servers, meaning that they correctly follow the steps of the protocol but may store the transcript of the messages exchanged and process the data received to try to learn more information than allowed.
 
 \begin{figure}[h]
  \centering
    \includegraphics[width=0.28\textwidth]{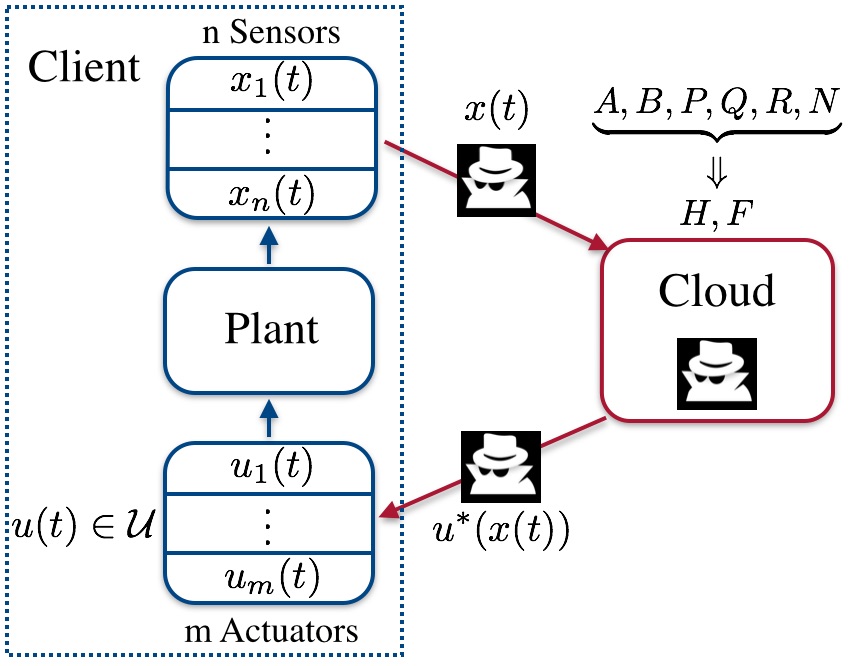}
  \caption{Unsecure MPC: the system model $A,B$, horizon $N$ and costs $P$, $Q$, $R$ are public. The state $x(t)$, control input $u^\ast(t)$ and input constraints $\mc U$ are privacy-sensitive. The red lines represent entities that can be eavesdropped.}
     \vspace{-0.2cm}
   \label{fig:unencrypted_data}
\end{figure}
 
To formalize the privacy objectives, we introduce the privacy definitions that we want our protocols to satisfy, described in~\cite[Ch.~3]{Goldreich03foundations|},~\cite[Ch.~7]{Goldreich04foundations||}. In what follows, $\{0,1\}^\ast$ defines a sequence of bits of unspecified length. Given a countable index set $I$, an ensemble $X = \{X_i\}_{i\in I}$, indexed by $I$, is a sequence of random variables $X_i$, for all $i\in I$. 

\begin{definition}\label{def:stat_ind}
The ensembles $X=\{X_n\}_{n\in \mathbb N}$ and $Y=\{Y_n\}_{n\in \mathbb N}$ are \textbf{statistically indistinguishable}, denoted $\stackrel{s}{\equiv}$, if for every positive polynomial $p$ and all sufficiently large $n$, the following holds, where $\Omega = \text{Supp}(X)\cup \text{Supp}(Y)$: 
\[1/2\sum_{\alpha\in\Omega} \left|\text{Pr}[X_n = \alpha] - \text{Pr}[Y_n = \alpha]\right|< 1/p(n).\]
\end{definition}
\vspace{\topsep}

Two ensembles are called computationally indistinguishable if no \textit{efficient} algorithm can distinguish between them. This is a weaker version of statistical indistinguishability. 

\begin{definition}\label{def:comp_ind}
The ensembles $X=\{X_n\}_{n\in\mathbb N}$ and $Y=\{Y_n\}_{n\in\mathbb N}$ are \textbf{computationally indistinguishable}, denoted~$\stackrel{c}{\equiv}$, if for every probabilistic polynomial-time algorithm $D:\{0,1\}^\ast\rightarrow \{0,1\}$, called the distinguisher, every positive polynomial $p$, and all sufficiently large $n$, the following holds:
\[\big |\text{Pr}_{x\leftarrow X_n} [D(x) = 1] - \text{Pr}_{y\leftarrow Y_n}[D(y) = 1] \big| < 1/p(n).\]
\end{definition}
\vspace{\topsep}

The definition of \textbf{\mbox{two-party} privacy} says that a protocol privately computes the functionality it runs if all information obtained by a party after the execution of the protocol, while also keeping a record of the intermediate computations, can be obtained only from the inputs and outputs of that party. 

\begin{definition}\label{def:view}
 Let $f\hspace{-0.2cm}:\hspace{-0.1cm}\left(\{0,1\}^\ast \right)^2\hspace{-0.1cm}\rightarrow\hspace{-0.1cm}\left(\{0,1\}^\ast \right)^2$ be a func-tionality, and $f_i(x_1,x_2)$ be the $i$th component of $f(x_1,x_2)$, $i=1,2$. 
Let $\Pi$ be a \mbox{two-party} protocol for computing $f$. The \textbf{view} of the $i$th party during an execution of $\Pi$ on the inputs $(x_1,x_2)$, denoted by $V^\Pi_i (x_1,x_2)$, is $(x_i,coins,m_1,\ldots,m_t)$, where $coins$ represents the outcome of the $i$th party's internal coin tosses, and $m_j$ represents the $j$th message it has received. For a deterministic functionality $f$, we say that $\Pi$ \textbf{privately computes} $f$ if there exist probabilistic polynomial-time algorithms, called \textbf{simulators}, denoted by $S_i$, such that:
 \[\{S_i(x_i,f_i(x_1,x_2)) \}_{x_{1,2}\in\{0,1\}^{\ast}} \stackrel{c}{\equiv} \{V_i^\Pi (x_1,x_2)\}_{x_{1,2}\in\{0,1\}^\ast}.\]
When the privacy is \mbox{one-sided}, i.e., only the part of the protocol executed by party 2 has to not reveal any information, the above equation has to be satisfied only for $i=2$.
 \end{definition}
 \vspace{\topsep}
 
The purpose of the paper is to design protocols with the functionality of Algorithm~\ref{alg:FGM} that satisfy Definition~\ref{def:view}. 
To this end, we use the encryption scheme defined in Section~\ref{sec:paillier}. Furthermore, we discuss in Section~\ref{sec:fixed-pointMPC} how we connect the domain of the inputs in Definiton~\ref{def:view} with the domain of real numbers needed for the MPC problem. 
In Sections~\ref{sec:CS} and~\ref{sec:SS}, we address two private cloud-MPC solutions that present a \mbox{trade-off} between the computational effort at the client and the total time required to compute the solution $u^\ast(t)$, which is analyzed in Section~\ref{sec:fixed-pointMPC}.

The MPC literature has focused on reducing the computational effort through computing a suboptimal solution to implicit MPC~\cite{Pannocchia11,Richter12}. 
Such time optimizations, i.e., stopping criteria, reveal information about the private data, such as how far is the initial point from the optimum or the difference between consecutive iterates. 
Therefore, in this work, we consider a given fixed number of iterations $K$. 

\section{Partially homomorphic cryptosystem}\label{sec:paillier}
Partially homomorphic encryption schemes can support additions between encrypted data, such as Paillier~\cite{Paillier99} and Goldwasser-Micali~\cite{GM82}, DGK~\cite{DGK07}, or multiplications between encrypted data, such as El Gamal~\cite{ElGamal84} and unpadded RSA~\cite{Rivest78data}. 

In this paper, we use the Paillier cryptosystem~\cite{Paillier99}, which is an asymmetric additively homomorphic encryption scheme. 
The message space for the Paillier scheme is $\mathbb Z_{N_\sigma}$, where $N_\sigma$ is a large integer that is the product of two prime numbers $p,q$. 
The pair of keys corresponding to this cryptosystem is $(pk,sk)$, where the public key is $pk= (N_\sigma,g)$, with $g\in\mathbb Z_{N_\sigma^2}$ having order $N_\sigma$ and the secret key is $sk = (\gamma,\delta)$:
\[\gamma = \text{lcm}(p-1,q-1), \delta = ((g^\gamma~\text{mod}~N_\sigma^2 -1)/N_\sigma)^{-1}~\text{mod}~N_\sigma.\]

For a message $a\in \mathbb Z_{N_\sigma}$, called plaintext, the Pailler encryption primitive is defined as: 
\[[[a]]:= g^ar^{N_\sigma}\bmod N_\sigma^2,~~\text{with $r$ random value in }\mathbb Z_{N_\sigma}.\] 
The encrypted messages are called ciphertexts.

A probabilistic encryption scheme, i.e., that takes random numbers in the encryption primitive, does not preserve the order from the plaintext space to the ciphertext space.

Intuitively, the additively homomorphic property states that there exists an operator $\oplus$ defined on the space of encrypted messages, with:
\[[[a]]\oplus[[b]] = [[a+b]],~~\forall a,b\in\mathbb Z_{N_\sigma},\]
where the equality holds in modular arithmetic, w.r.t. $N_\sigma^2$. 
Formally, the decryption primitive is a homomorphism between the group of ciphertexts, with the operator $\oplus$ and the group of plaintexts with addition~$+$. 

The scheme also supports subtraction between ciphertexts, and multiplication between a plaintext and an encrypted message, obtained by adding the encrypted messages for the corresponding (integer) number of times: \mbox{$b\otimes[[a]]=[[ba]]$}. 
We will use the same notation to denote encryptions, additions and multiplication by vectors and matrices. 

Proving privacy in the \mbox{semi-honest} model of a protocol that makes use of cryptosystems involves the concept of semantic security. 
Under the assumption of decisional composite residuosity~\cite{Paillier99}, the Paillier cryptosystem is semantically secure and has indistinguishable encryptions, which, in essence, means that an adversary cannot distinguish between the ciphertext $[[a]]$ and a ciphertext $[[b]]$ based on the messages $a$ and $b$~\cite[Ch.~5]{Goldreich04foundations||}.

\begin{definition}\label{def:semantic}
An encryption scheme with encryption primitive $E(\cdot)$ is \textbf{semantically secure} if for every probabilistic polynomial-time algorithm, A, there exists a probabilistic polynomial-time algorithm $A'$ such that for every two polynomially bounded functions $f,h:\{0,1\}^\ast\rightarrow \{0,1\}^\ast$ and for any probability ensemble $\{X_n\}_{n\in\mathbb N}$, $|X_n|\leq poly(n)$, for any positive polynomial $p$ and sufficiently large $n$: 
\begin{align*}
\text{Pr}& \left[ A(E(X_n),h(X_n)) = f(X_n)\right] < \\ \text{Pr}& \left[ A'(h(X_n)) = f(X_n)\right] + 1/p(n).
\end{align*}
\end{definition}
\vspace{\topsep}

\emph{Cloud-based Linear Quadratic Regulator}: 
We provide a simple example of how the Paillier encryption can be used in a private control application. If the problem is unconstrained, i.e., $\mc U=\mathbb R^m$, a stabilizing controller can be computed as a linear quadratic regulator~\cite[Ch.~8]{MPCbook17}. 
Such a controller can be computed only by one server. The client sends the encrypted state $[[x(t)]]$ to the server, which recursively computes the plaintext control gain and solution to the Discrete Riccati Equations, with $P_N = P$:
\begin{align*}
	&F_k = -(B^\intercal P_{k+1}B + R)^{-1} B^\intercal P_{k+1}A,\quad k=N-1,\ldots,0\\
	&P_k = A^\intercal P_{k+1}A + Q + A^\intercal P_{k+1} B F_k,\\
	&[[u^\ast(t)]] = F_0\otimes[[x(t)]].
\end{align*}
The server obtains the encrypted control input and returns it to the client, which decrypts and inputs it to the system.

To summarize, PHE allows a party that does not have the private key to perform linear operations on encrypted integer data. 
For instance, a cloud-based Linear Quadratic Controller can be computed entirely by one server, because the control action is linear in the state. 
Nonlinear operations are not supported within this cryptosystem, but can be achieved with communication between the party that has the encrypted data and the party that has the private key, and we will use this in Sections~\ref{sec:CS} and~\ref{sec:SS}. 

%%%%%%%%%%%%%%%%%%%%%%%%%%%%%%%%%%%%%%%%%%%%%%%%%%%%%%%%%%%%%%%%%%%%%%%%%%%%%%%%

\section{Client-Server architecture}\label{sec:CS}
To be able to use the Paillier encryption, we need to represent the messages on a finite set of integers, parametrized by $N_\sigma$, i.e., each message is an element in $\mathbb Z_{N_\sigma}$. 
Usually, the values less than $N_\sigma/3$ are interpreted to be positive, the numbers between $2N_\sigma/3$ and $N_\sigma$ to be negative, and the rest of the range allows for overflow detection. 
In this section and Section~\ref{sec:SS}, we consider a \mbox{fixed-point} representation of the values and perform implicit multiplication steps to obtain integers and division steps to retrieve the true values. 
We analyze the implications of the \mbox{fixed-point} representation over the MPC solution in Section~\ref{sec:fixed-pointMPC}.

\textit{Notation:} Given a real quantity $x\in\mathbb R$, we use the notation $\bar x$ for the corresponding quantity in \mbox{fixed-point} representation on one sign bit, $l_i$ integer and $l_f$ fractional bits. 

We introduce a \mbox{client-server} model, depicted in Figure~\ref{fig:client-server}. 
We present an interactive protocol that privately computes the control input for the client, while maintaining the privacy of the state in Protocol~\ref{pro:FGM_CS}. 
The Paillier encryption is not order preserving, so the projection operation cannot be performed locally by the server. 
Hence, the server sends the encrypted iterate $[[t_k]]$ for the client to project it. Then, the client encrypts the feasible iterate and sends it back to the cloud.

\begin{figure}[h]
  \centering
    \includegraphics[width=0.3\textwidth]{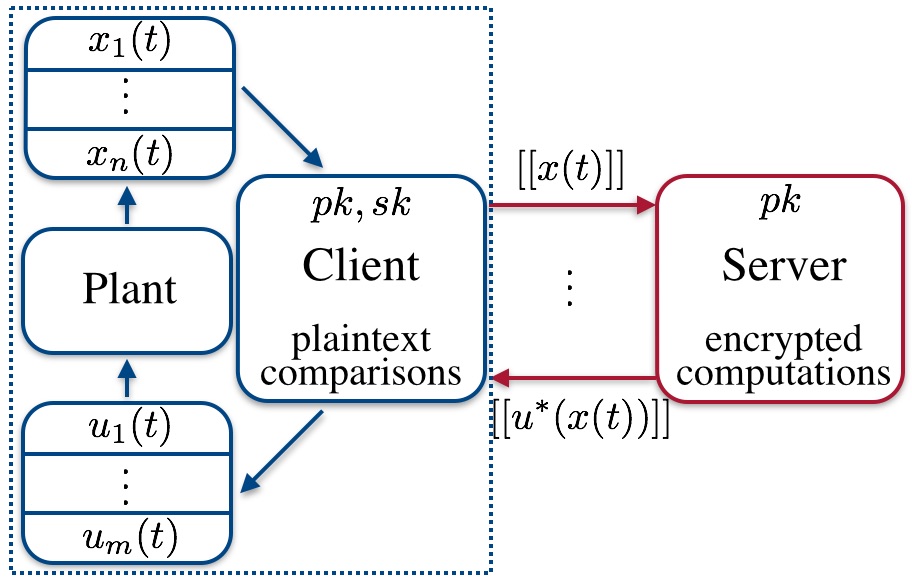}
       \vspace{-0.2cm}
    \caption{Private \mbox{client-server} MPC setup for a plant.}
   \label{fig:client-server}
\end{figure}

We drop the $\bar{(\cdot)}$ from the iterates in order to not burden the notation.

\begin{protocol}[Encrypted projected Fast Gradient Descent in a \mbox{client-server} architecture] \label{pro:FGM_CS}
\begin{algorithmic}[1]
\small
 \Require{$C$: $x(t), K_c,K_w, l_i, l_f,\bar{\mc U},pk,sk$;~$S$: $\bar H_f, \bar F,\bar\eta, K_c,  K_w, l_i, $ $ l_f,pk$, cold-start, $[[U_{w}]]$}
 \Ensure{$C$: $u = (U_K(x(t)))_{1:m}$}
  \State $C$: Encrypt and send $[[x(t)]]$ to $S$
  \If{cold-start}
	\State  $S$: $[[U_0]] = [[\mb 0_{Nm}]]$; 
	$C,S$: $K \leftarrow K_c$
 	\Else
	\State $S$: $[[U_0]] = \big[ [[(U_{w})_{m+1:Nm}]],[[\mb 0_m]]\big]$;
	$C,S$: $K \leftarrow K_w$
 \EndIf
 \State $S$: $[[z_0]] = [[U_0]]$
\For{k=0\ldots,K-1}
 	\State $S$: $[[ t_k]] = (\mathbf I_{Nm} - \bar H_f)\otimes [[z_k]] \oplus (- \bar F_f^\intercal) \otimes [[x(t)]]$ and send it to $C$
	\State $C$: Decrypt $t_k$ and truncate to $l_f$ fractional bits
	\State $C$: $U_{k+1} = \Pi_{\bar{\mc U}}(t_k)$ \Comment Projection on $\bar{\mc U}$
	\State $C$: Encrypt and send $[[U_{k+1}]]$ to $S$
	\State $S$: $[[z_{k+1}]] = (1+\bar \eta)\otimes [[U_{k+1}]] \oplus (-\bar\eta) \otimes[[U_{k}]]$
 \EndFor
 \State $C$: Decrypt and output $u = (U_K)_{1:m}$
\end{algorithmic}
\end{protocol}

\begin{theorem}\label{thm:client-server}
	Protocol~\ref{pro:FGM_CS} achieves privacy as in Definition~\ref{def:view} with respect to a semi-honest server.
\end{theorem}

\begin{proof}
The initial value of the iterate does not give any information to the server about the result, as the final result is encrypted and the number of iterations is a priori fixed. 
The view of the server, as in Definition~\ref{def:view}, is composed of the server's inputs, the messages received $\{[[U_k]]\}_{k=0,\ldots,K}$, which are all encrypted, and no output. 
We construct a simulator which replaces the messages with random encryptions of corresponding length. 
Due to the semantic security (see Definition~\ref{def:semantic}) of the Paillier cryptosystem, which was proved in~\cite{Paillier99}, the view of the simulator is computationally indistinguishable from the view of the server.
\end{proof}

%%%%%%%%%%%%%%%%%%%%%%%%%%%%%%%%%%%%%%%%%%%%%%%%%%%%%%%%%%%%%%%%%%%%%%%%%%%%%%%%
\section{Two-server architecture}~\label{sec:SS}
Although in Protocol~\ref{pro:FGM_CS}, the client needs to store and process substantially less data than the server, the computational requirements might be too stringent for large values of $K$ and $N_\sigma$. 
In such a case, we outsource the problem to two servers, and only require the client to encrypt $x(t)$ and send it to one server and decrypt the received the result~$[[u^\ast]]$. 
In this setup, depicted in Figure~\ref{fig:two-server}, the existence of two \mbox{non-colluding} servers is assumed. 

In Figure~\ref{fig:two-server} and in Protocol~\ref{pro:FGM_SS}, we will denote by $[[\cdot]]$ a message encrypted with $pk_1$ and by $[\{\cdot\}]$ a message encrypted by $pk_2$. 
The reason we use two pairs of keys is so the client and support server do not have the same private key and do not need to interact.

\begin{figure}[h]
  \centering
    \includegraphics[width=0.325\textwidth]{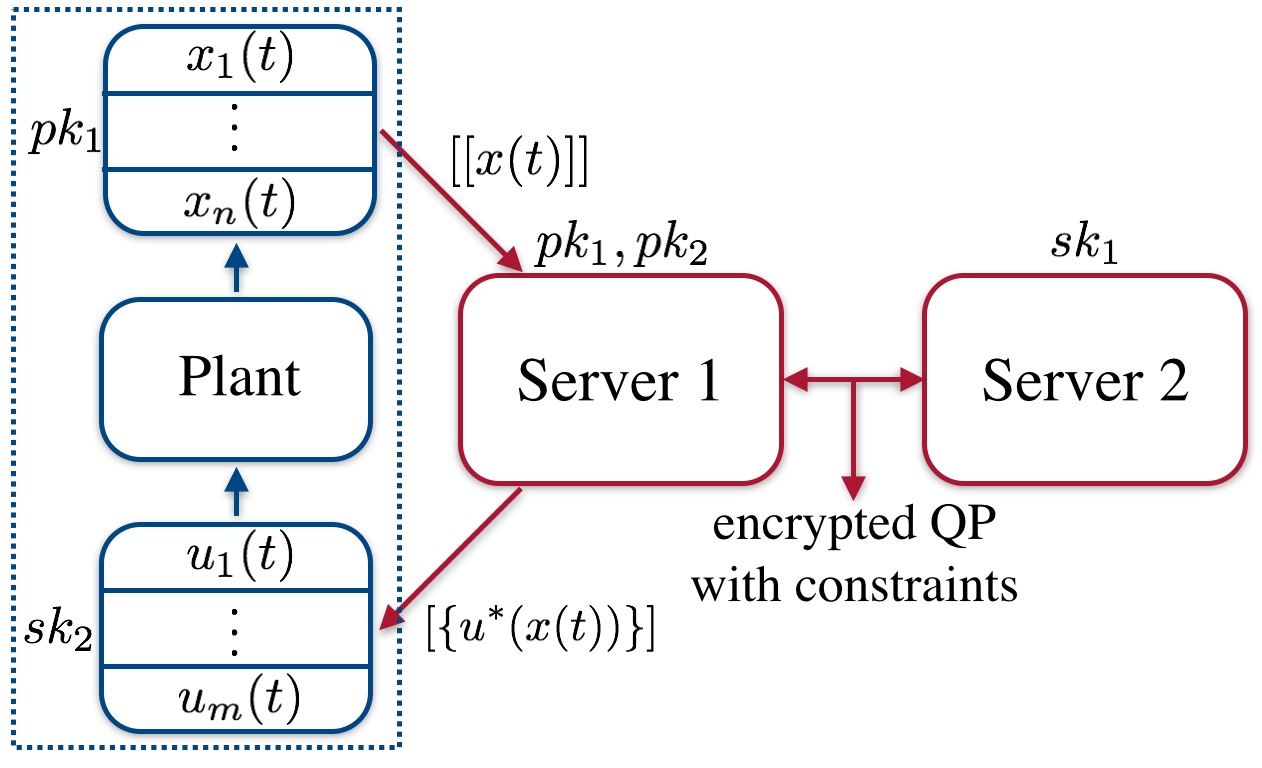}
       \vspace{-0.2cm}
   \caption{Private \mbox{two-server} MPC setup for a plant.}
   \label{fig:two-server}
\end{figure}

As before, we need an interactive protocol to achieve the projection. We use the DGK comparison protocol, proposed in~\cite{DGK07,DGK09correction}, such that, given two encrypted values of $l$ bits $[[a]],[[b]]$ to $S_1$, after the protocol, $S_2$, who has the private key, obtains a bit $(\beta = 1)\equiv(a\leq b)$, without finding anything about the inputs. 
Moreover, $S_1$ finds nothing about~$\beta$. We augment this protocol by introducing a step before the comparison in which $S_1$ randomizes the order of the two values to be compared, such that $S_2$ does not know the significance of~$\beta$ with respect to the inputs. 
Furthermore, by performing a blinded exchange, $S_1$ obtains the minimum (respectively, maximum) value of the two inputs, without any of the two servers knowing what the result is. The above procedure is performed in lines 11--16 in Protocol~\ref{pro:FGM_SS}. More details can be found in~\cite{Alexandru2018cloudQPHE}.

The comparison protocol works with inputs that are represented on $l$ bits. 
The variables we compare are results of additions and multiplications, which can increase the number of bits, thus, we need to ensure that they are represented on $l$ bits before inputting them to the comparison protocol. 
This introduces an extra step in line~10 in which $S_1$ communicates with $S_2$ in order to obtain the truncation of the comparison inputs: $S_1$ adds noise to $t_k$ and sends it to $S_2$ which decrypts it, truncates the result to $l$ bits and sends it back. $S_1$ then subtracts the truncated noise. 

In order to guarantee that $S_2$ does not find out the private values after decryption, $S_1$ adds a sufficiently large random noise to the private data. 
The random numbers in lines~13 and 19 are chosen from $(0,2^{l+\lambda_{\sigma}})\cap \mathbb Z_{N_\sigma}$, which ensures the statistical indistinguishability between the sum of the random number and the private value and a random number of equivalent length~\cite{Bost15}, where $\lambda_{\sigma}$ is the statistical security parameter.

\begin{protocol}[Encrypted projected Fast Gradient Descent in a \mbox{two-server} architecture] \label{pro:FGM_SS}
\begin{algorithmic}[1]
\small
 \Require{$C:x(t),pk_{1,2},sk_2;S_1:\bar H_f, \bar F,\bar\eta,K_c,K_w, l_i, l_f, pk_{1,2},$ $\text{cold-start},[[U_{w}]]; S_2: K_c,K_w,  l_i, l_f, pk_{1,2},sk_1, \text{cold-start}$}
 \Ensure{$C$: $u = (U_K(x(t)))_{1:m}$}
  \State $C$: Encrypt and send $[[x(t)]],[[h_u]],[[-l_u]]$ to $S_1$
  \If{cold-start}
	\State $S_1$: $[[U_0]] \leftarrow [[\mb 0_{Nm}]]$; 
$S_1,S_2$:$K \leftarrow K_c$
 	\Else
	\State $S_1$: $[[U_0]] \leftarrow \big[ [[(U_{w})_{m+1:Nm}]],[[\mb 0_m]]\big]$; $S_1,S_2$:$K \leftarrow K_w$
 \EndIf
 \State $S_1$: $[[z_0]]\leftarrow[[U_0]]$
\For{k=0\ldots,K-1}
 	\State $S_1$: $[[t_k]] \leftarrow (\mathbf I_{Nm} - \bar H_f)\otimes [[z_k]] \oplus (- \bar F_f^\intercal)\otimes [[x(t)]]$
	\State $S_1$: $[[t_k]]\leftarrow$ truncate $[[t_k]]$
	\State $S_1$: $a_k,b_k\leftarrow$ randomize $[[t_k]],[[h_u]]$
	\State $S_1,S_2$: DGK s.t. $S_2$ obtains $(\beta_k = 1)\equiv(a_k\leq b_k)$
	\State $S_1$: Pick $r_k,s_k$ and send $[[a_k]]\oplus[[r_k]],[[b_k]]\oplus[[s_k]]$ to $S_2$ 
	\State $S_2$: Send back $[[\beta_k]]$ and $[[v_k]] \leftarrow [[a_k+r_k]]\oplus[[0]]$ if $\beta_k=1$ or $[[v_k]] \leftarrow[[b_k+s_k]]\oplus[[0]]$ if $\beta_k=0$
	\State $S_1$: $[[U_{k+1}]]\leftarrow [[v_k]] \oplus s_k\otimes([[\beta_k]]\oplus[[-1]]) \oplus r_k\otimes [[\beta_k]]$ \Comment{$U_{k+1}\leftarrow\min(t_k,h_u)$}
	\State $S_1,S_2$: Redo 11--15 to get $[[U_{k+1}]]\leftarrow\max(U_{k+1},-l_u)$
	\State $S_1$: $[[z_{k+1}]] \leftarrow (1+\bar \eta)\otimes [[U_{k+1}]] \oplus (-\bar\eta) \otimes [[U_{k}]]$
 \EndFor
 \State $S_1$: Pick $\rho$ and send $[[(U_K)_{1:m}]]\oplus[[\rho]]$ to $S_2$
 \State $S_2$: Decrypt, encrypt with $pk_2$ and send to $S_1$: $[\{u+\rho\}]$
  \State $S_1$: $[\{u\}]\leftarrow [\{u+\rho\}]\oplus[\{-\rho\}]$ and send it to $C$
 \State $C$: Decrypt and output $u$
\end{algorithmic}
\end{protocol}

\begin{theorem}\label{thm:two-server}
Protocol~\ref{pro:FGM_SS} achieves privacy as in Definition~\ref{def:view}, as long as the two \mbox{semi-honest} servers do not collude.
\end{theorem}

\begin{proof}
The view of $S_1$ is composed by its inputs and exchanged messages, and no output. All the messages the first server receives are encrypted (the same holds for the subprotocol DGK). 
Furthermore, in line~14, an encryption of zero is added to the quantity $S_1$ receives such that the encryption is re-randomized and $S_1$ cannot recognize it. 
Due to the semantic security of the cryptosystems, the view of $S_1$ is computationally indistinguishable from the view of a simulator which follows the same steps as $S_1$, but replaces the incoming messages by random encryptions of corresponding length.

The view of $S_2$ is composed by its inputs and exchanged messages, and no output. Apart from the comparison bits, the latter are always blinded by noise that has at least $\lambda_{\sigma}$ bits more than the private data being sent. 
For $\lambda_{\sigma}$ chosen appropriately large (e.g.~100 bits~\cite{Bost15}), the following is true: $a+r\stackrel{s}{\equiv} r'$, where $a$ is a value of $p$ bits, $r$ is the noise chosen uniformly at random from $(0,2^{p+\lambda_{\sigma}})\cap \mathbb Z_{N_\sigma}$ and $r'$ is a value chosen uniformly at random from $(0,2^{p+\lambda_{\sigma}+1})\cap \mathbb Z_{N_\sigma}$. In the DGK subprotocol, a similar blinding is performed, see~\cite{Veugen12}.

Crucially, the noise selected by $S_1$ is different at each iteration. Hence, $S_2$ cannot extract any information by combining messages from multiple iterations, as they are always blinded by a different large enough noise. 
Moreover,~the randomization step in line~11 ensures that $S_2$ cannot infer anything from the values of~$\beta_k$, as the order of the inputs is unknown. 
Thus, we construct a simulator that follows the same steps as $S_2$, but instead of the received messages, it randomly generates values of appropriate length, corresponding to the blinded private values, and random bits corresponding to the comparison bits. 
The view of such a simulator will be computationally indistinguishable from the view of $S_2$.
\end{proof}

\begin{remark}
One can expand Protocols~\ref{pro:FGM_CS} and~\ref{pro:FGM_SS} over multiple time steps, such that $U_{0}$ is obtained from the previous iteration and not given as input, and formally prove their privacy. 
The fact that the state will converge to a neighborhood of the origin is public knowledge, and is not revealed by the execution of the protocol. A more detailed proof that explicitly constructs the simulators can be found in~\cite{Alexandru2018cloudQPHE}. 
\end{remark}

Through communication, encryption and statistical blinding, the two servers can privately compute nonlinear operations. However, this causes an increase in the computation time due to the extra encryptions and decryptions and communication rounds, as will be pointed out in Section~\ref{sec:numerical}.

%%%%%%%%%%%%%%%%%%%%%%%%%%%%%%%%%%%%%%%%%%%%%%%%%%%%%%%%%%%%%%%%%%%%%%%%%%%%%%%%
\section{\mbox{Fixed-point} arithmetics MPC}\label{sec:fixed-pointMPC}
The values that are encrypted or added to or multiplied with encrypted values have to be integers. We consider \mbox{fixed-point} representations with one sign bit, $l_i$ integer bits and $l_f$ fractional bits and multiply them by $2^{l_f}$ to obtain integers. 

Working with \mbox{fixed-point} representations can lead to overflow, quantization and arithmetic \mbox{round-off} errors. 
Thus, we want to compute the deviation between the fixed-point solution and optimal solution of Algorithm~\ref{alg:FGM}. 
The bounds on this deviation can be used in an offline step prior to the online computation to choose an appropriate \mbox{fixed-point} precision for the performance of the system. 

Consider the following assumption:
\begin{assumption}\label{assum:stability}
The number of fractional bits $l_f$ and constant $c\geq 1$ are chosen large enough such that:
\begin{enumerate}
	\item[(i)] $\bar{\mc U}\subseteq \mc U$: the \mbox{fixed-point} precision solution is still feasible. 
	\item[(ii)] The eigenvalues of the \mbox{fixed-point} representation $\bar H_f$ are contained in the set $(0,~1]$, where 
	\[H_f := \bar H/ (c\bar L) ~\text{and}~\bar L := \lambda_{max}(\bar H).\]
	The constant $c$ is required in order to overcome the possibility that $\overline{(1/\bar L)\bar H}$ has the maximum eigenvalue larger than 1 due to \mbox{fixed-point} arithmetic errors.
	\item[(iii)] The \mbox{fixed-point} representation of the step size satisfies:
	\[0\leq \left(\sqrt{\kappa(\bar H)}-1\right)\Big/\left(\sqrt{\kappa(\bar H)}+1\right) \leq \bar \eta < 1.\]
\end{enumerate}
\end{assumption}

Item~(i) ensures that the feasibility of the \mbox{fixed-point} precision solution is preserved, item~(ii) ensures that the strong convexity of the \mbox{fixed-point} objective function still holds and item~(iii) ensures that the \mbox{fixed-point} step size is such that the FGM converges.
\vspace{\topsep}

\noindent\emph{Overflow errors:} 
Bounds on the infinity-norm on the \mbox{fixed-point} dynamic quantities of interest in Algorithm~\ref{alg:FGM} were derived in~\cite{Jerez14} for each iteration $k$, and depend on a bounded set $\mc X_0$ such that $x(t)\in\mc X_0$ and $\bar x(t) \in \bar{\mc X_0}$:
\begin{align*}
||\bar U_{k+1}||_\infty &\leq \max \{||\bar l_u||_\infty, ||\bar h_u||_\infty\} = R_{\bar{\mc U}}\\
||\bar z_{k+1}||_\infty &\leq (1+ 2\bar \eta) \mc R_{\bar{\mc U}}: = \zeta,\\
||\bar t_k||_\infty &\leq ||\mb{I}_{Nm} - \bar H_f||_\infty \zeta + ||\bar F_f||_\infty \mc R_{\bar{\mc X_0}},
\end{align*}
where $F_f = \bar F/(c\bar L)$ and $\mc R_{\mc S}$ represents the radius of a set $\mc S$ w.r.t. the infinity norm. 
We select from these bounds the number of integer bits $l_i$ such that there is no overflow.

\subsection{Difference between real and \mbox{fixed-point} solution}
We want to determine a bound on the error between the \mbox{fixed-point} precision solution and the real solution of the MPC problem~\eqref{eq:mpc(i)}. The total error is composed of the error induced by having \mbox{fixed-point} coefficients and variables in the optimization problem, and by the round-off errors. 
Specifically, denote by $U_K$ the solution in exact arithmetic of the MPC problem~\eqref{eq:mpc(i)} obtained after $K$ iterations of Algorithm~\ref{alg:FGM}. 
Furthermore, denote by $\tilde U_K$ the solution obtained after $K$ iterations but with $H,F,x(t),\mc U,L,\eta$ replaced by their \mbox{fixed-point} representations. 
Finally, denote by $\bar U_K$ the solution of Protocols~\ref{pro:FGM_CS} and~\ref{pro:FGM_SS} after $K$ iterations, where the iterates $[[t_k]],[[U_k]]$ have~\mbox{fixed-point} representation, i.e, truncations are performed. 
We obtain the following upper bound on the difference between the solution obtained on the encrypted data and the nominal solution of the implicit MPC problem~\eqref{eq:mpc(i)} after $K$ iterations:
\[||\bar U_K - U_K||_2 \leq ||\tilde U_K - U_K||_2+ ||\bar U_K - \tilde U_K||_2.\]

\subsubsection{Quantization errors}
We will use the following observation to investigate the quantization error bounds. Define $\epsilon_a = \bar a - a$ and $\epsilon_b = \bar b - b$. Then: 
\[\bar a\bar b - ab = \bar a\bar b - \bar a b + \bar a b - ab = \epsilon_a b + \bar a \epsilon_b.\]

Consider problem~\eqref{eq:mpc(i)} where the coefficients are replaced by the \mbox{fixed-point} representations of the matrices $H/(cL),F/(cL)$, of the vector $x(t)$ and of the set $\mc U$, but otherwise the iterates $\tilde U_k,\tilde t_k, \tilde z_k$ are real values. 
Now, consider iteration $k$ of the projected FGM. The errors induced by quantization of the coefficients between the original iterates and the approximation iterates will be:
\begin{align}\label{eq:errors_q}
\begin{split}
	\tilde t_k - t_k =& -\epsilon_{H_f}z_k + (\mb I_{Nm}-\bar H_f)\epsilon_{z,k}-\epsilon_{Fx}\\
	\xi^q_{k+1} :=&~\tilde U_{k+1} - U_{k+1} = D_k^q(\tilde t_k - t_k)	\\
	 \tilde z_{k+1} - z_{k+1} =&~ \epsilon_\eta\Delta U_k + (1+\bar \eta)\xi^q_{k+1} - \bar\eta \xi^q_{k},
\end{split}
\end{align}
where we used the notation: 
$\Delta U_k = U_{k+1}-U_{k}$; \mbox{$\epsilon_\eta = \bar \eta - \eta$}; $\epsilon_{H_f} =  \bar H_f - H/(cL)$; 
$\epsilon_{Fx} = \bar F_f^\intercal \bar x(t)-F^\intercal x(t)/(c L) = \epsilon_{F_f}^\intercal x(t) + \bar F_f \epsilon_x$; 
$\epsilon_x = \bar x(t) - x(t)$; $\epsilon_{F_f} = \bar F_f - F/(cL)$. 

The error between $\tilde U_{k+1}$ and $U_{k+1}$ is reduced from $\tilde t_k - t_k$ due to the projection on the hyperbox. Hence, to represent $\xi^q_{k+1}$ in~\eqref{eq:errors_q}, we multiply $\tilde t_k - t_k$ by the diagonal matrix $D_k^q$ that has positive elements at most one. 

We set $\xi^q_{-1}=\xi^q_{0}$. From~\eqref{eq:errors_q}, we derive a recursive iteration that characterizes the error of the primal iterate, for $k=0,\ldots,K$, which we can write as a linear system:
\begin{align*}
	\tilde A(D_k^q):=& \left[ \begin{matrix}(1+\bar \eta)D_k^q(\mb I_{Nm}  - \bar H_f) & -\bar\eta D_k^q(\mb I_{Nm}  - \bar H_f)\\ \mb I_{Nm}  & \mb 0_{Nm}\end{matrix}\right]\\
	\tilde B(D_k^q):=& \left[ \begin{matrix}-\epsilon_{H_f}D_k^q & \epsilon_\eta D_k^q(\mb I_{Nm} - \bar H_f)\\ \mb 0_{Nm}  & \mb 0_{Nm}   \end{matrix}\right] \numberthis\label{eq:err_q_sys}\\
	\left[ \begin{matrix}\xi^q_{k+1}\\ \xi^q_k\end{matrix}\right] =& \tilde A(D_k^q) \left[ \begin{matrix}\xi^q_{k}\\ \xi^q_{k-1}\end{matrix}\right] + \tilde B(D_k^q)\left[ \begin{matrix}z_k\\ \Delta U_{k-1} \end{matrix}\right] - \epsilon_{Fx}.
\end{align*}

We choose this representation in order to have a relevant error bound in Theorem~\ref{thm:quantization} that shrinks to zero as the number of fractional bits grows. 
In the following, we find an upper bound of the error using $\tilde A := \tilde A(\mb I_{Nm})$ and $\tilde B := \tilde B(\mb I_{Nm})$.

\begin{theorem}\label{thm:quantization}
Under Assumption~\ref{assum:stability}, the system defined by~\eqref{eq:err_q_sys} is bounded. Furthermore, the norm of the error between the primal iterates of the original problem and of the problem with quantized coefficients is bounded by:
\begin{align*}
\begin{split}
&||\xi^q_{k}||_2 \leq \left|\left|E\tilde A^k \right|\right|_2 \left|\left|\left[ \begin{matrix}\xi^q_{0}\\ \xi^q_{-1}\end{matrix}\right] \right|\right|_2
 + \gamma \sum_{l=0}^{k-1} \left|\left| E\tilde A^{k-1-l}\tilde B\right|\right|_2 + \\
 &+ \zeta\sum_{l=0}^{k-1} \left|\left| E\tilde A^{k-1-l}\right|\right|_2 =: \epsilon_1;~~
 \gamma = (3+2\bar\eta)\sqrt{Nm}\mc R_{\bar{\mc U}},\\
 &\zeta = ||\epsilon_{F_f}||_2 \mc R^2_{\mc X_0} + 2^{-l_f}\sqrt{n}||\bar F_f||_2,
\end{split}
\end{align*}
where $E = \left[\mb I_{Nm}~~ \mb0_{Nm}\right]$, $\mc R^2_{\mc X_0} $ is the radius of the compact set $\mc X_0$ w.r.t. the 2-norm and $\mc R_{\bar{\mc U}}=\max \{||l_u||_\infty, ||h_u||_\infty\}$.
\end{theorem}

\begin{proof}
	The inner stability of the system is given by the fact that $\tilde A$ has spectral radius $\rho(\tilde A)<1$ which is proven in Lemma~1 in~\cite{Jerez14}. The same holds for $\tilde A(D_k^q)$. Since we want to give a bound of the error in terms of computable values, we use the fact that $||\tilde A(D_k^q)||_2\leq ||\tilde A||_2$ (resp., $||\tilde B(D_k^q)||_2\leq ||\tilde B||_2$) and express the bounds in terms of the latter.
	
	From~\eqref{eq:err_q_sys}, one can obtain the following expression for the errors at time $k$ and $k-1$, for $k=0,\ldots,K-1$:
\begin{align*}
	\left[ \begin{matrix}\xi^q_{k}\\ \xi^q_{k-1}\end{matrix}\right] &\leq \tilde A^{k} \left[ \begin{matrix}\xi^q_{0}\\ \xi^r_{-1}\end{matrix}\right] + 
	\sum_{l=0}^{k-1}\tilde A^{k-1-l}\left(\tilde B  \left[ \begin{matrix}z_l\\ \Delta U_{l-1}\end{matrix}\right] - \epsilon_{Fx} \right),
\end{align*}
and the first term goes to zero as $k\rightarrow \infty$. 
We multiply this by $E = \left[\mb I_{Nm}~\mb0_{Nm}\right]$ to obtain the expression of $||\xi^q_{k}||$.

	Subsequently, for any $0\leq k\leq K-1$:
\begin{equation*}
\begin{aligned}
&\left | \left | \left[ z_k^\intercal ~~ \Delta U_{k-1}^\intercal\right]^\intercal \right| \right|_2 = ||z_k ||_2 + ||\Delta U_{k-1}||_2,\\
&\leq ||U_{k} + \bar\eta \Delta U_{k-1}||_2 + ||\Delta U_k||_2 \leq \mc R_{\bar{\mc U}} + 2(1+\bar\eta)\mc R_{\bar{\mc U}}\\
& \leq (3+2\bar\eta)\sqrt{Nm (\max\limits_i \{l_u^i,h_u^i\})^2} :=\gamma\\
&||\epsilon_{Fx}||_2 \leq  ||\epsilon_{F_f}||_2 || x(t)||_2 + ||\bar F_f||_2 || \epsilon_x||_2\\
&\leq  ||\epsilon_{F_f}||_2 \mc R^2_{\mc X_0} + 2^{-l_f}\sqrt{n}||\bar F_f||_2 :=\zeta. %\qquad \qquad \qquad\qquad\QEDclosed
\end{aligned}
\end{equation*}
\end{proof}

One can eliminate the initial error $\xi^q_0$ and its effect by choosing in both exact and \mbox{fixed-point} coefficient-FGM algorithms the initial iterate to be represented on $l_f$ fractional bits. Therefore, only the persistent noise counts. 

\begin{remark}
In primal-dual algorithms, the maximum values of the dual variables corresponding to the complicating constraints cannot be bounded a priori, i.e., we cannot give overflow or coefficient quantization error bounds. This justifies our focus on a problem with only simple input constraints. 
The work in~\cite{Patrinos13} considers the bound on the dual variables as a parameter that can be tuned by the user. 
\end{remark}
\vspace{\topsep}

\subsubsection{Arithmetic round-off errors}
Let us now investigate the error between the solution of the previous problem and the solution of the \mbox{fixed-point} FGM corresponding to Protocols~\ref{pro:FGM_CS} and~\ref{pro:FGM_SS}. 
The encrypted values do not necessarily maintain the same number of bits after operations, so we will consider round-off errors where we perform truncations. This happens in line~10 in both protocols. 
In this case, we obtain similar results to~\cite{Jerez14}, where the quantization errors were not analyzed, i.e., as if the nominal coefficients of the problem were represented with $l_f$ fractional bits from the problem formulation. Consider iteration $k$ of the projected FGM. 
The errors due to \mbox{round-off} between the primal iterates of the two solutions will be:
\begin{align*}
	\bar t_k - \tilde t_k =& (\mb I_{Nm} - \bar H_f)(\bar z_k - \tilde z_k) + \epsilon'_{t,k}\\
	\xi^r_{k+1} :=& \bar U_{k+1} - \tilde U_{k+1} = D_k^r(\bar t_k - t_k)      \numberthis \label{eq:errors_r}\\
	\bar z_{k+1} - \tilde z_{k+1} =& (1+\bar \eta)\xi^r_{k+1} - \bar\eta \xi^r_{k}.
\end{align*}
Again, the projection on the hyperbox reduces the error, so $D_k^r$ is a diagonal matrix with positive elements less than one. For Protocol~\ref{pro:FGM_CS}, the \mbox{round-off} error due to truncation is \mbox{$(\epsilon'_{t,k})^i \in [-Nm2^{-l_f},~0]$}, $i=1,\ldots,Nm$. 
The encrypted truncation step in Protocol~\ref{pro:FGM_SS} introduces an extra term due to blinding, making \mbox{$(\epsilon'_{t,k})^i\in[-(1+Nm)2^{-l_f},~2^{-l_f}]$}.

We set set $\xi^r_{-1}=\xi^r_{0}$. From~\eqref{eq:errors_r}, we can derive a recursive iteration that characterizes the error of the primal iterate, which we can write as a linear system, with $\tilde A(\cdot)$ as before:
\begin{align}\label{eq:err_r_sys}
\begin{split}
	\left[ \begin{matrix}\xi^r_{k+1}\\ \xi^r_k\end{matrix}\right] =& \tilde A(D_k^r) \left[ \begin{matrix}\xi^r_{k}\\ \xi^r_{k-1}\end{matrix}\right] + D_k^r \epsilon'_{t,k}.
\end{split}
\end{align}

\begin{theorem}\label{thm:round-off}
Under Assumption~\ref{assum:stability}, the system defined by~\eqref{eq:err_r_sys} is bounded. Furthermore, the norm of the error of the primal iterate is bounded by:
\begin{align*}
||\xi^r_k||_2 \leq& \left|\left|E\tilde A^k \right|\right|_2 \left|\left|\left[ \begin{matrix}\xi^r_{0}\\ \xi^r_{-1}\end{matrix}\right] \right|\right|_2 + \gamma'\sum_{l=0}^{k-1} \left|\left| E\tilde A^{k-1-l}\right|\right|_2 =: \epsilon_2,\\
\gamma_{CS}'=& 2^{-l_f}(Nm)^\frac{3}{2}; \gamma_{SS}'=2^{-l_f}\sqrt{Nm}(1+Nm).
\end{align*}
\end{theorem}
\vspace{\topsep}

The proof is straightforward. 

As before, one can eliminate the initial error $\xi^r_0$ and its effect by choosing the same initial iterate represented on $l_f$ fractional bits for both problems. 

\begin{remark}
As $l_f\rightarrow \infty$, $\epsilon_1\rightarrow 0$ and $\epsilon_2\rightarrow 0$. The persistent noise in~\eqref{eq:err_q_sys} and~\eqref{eq:err_r_sys}, which is composed by quantization errors and round-off errors, becomes zero when the number of fractional bits mimics a real variable. This makes systems~\eqref{eq:err_q_sys} and~\eqref{eq:err_r_sys} input-to-state stable.
\end{remark}

\subsection{Choice of error level}
For every instance of problem~\eqref{eq:mpc}, the error bound $\epsilon:=\epsilon_1+\epsilon_2$ can be computed as a function of the number of fractional bits~$l_f$. 
We can incorporate these errors as a bounded disturbance $d(\cdot)$ in system~\eqref{eq:system}: $d(t) = B\xi(t)$, where $\xi(t) = u^\ast(t) - \bar u(t)$. 
Then, we can design the terminal cost as described in~\cite{Pannocchia11}, such that the controller achieves inherent robust stability to process perturbations $||d(t)||_2\leq \delta$, and choose: $\epsilon\leq \frac{\delta}{||B||}$. 
Alternatively, we can incorporate the error in the suboptimality of the cost and assume that we obtain a cost \mbox{$\bar J_{N}(x(t),\bar U_K) \leq J_N^\ast(x(t)) + \epsilon'$}, and compute $\epsilon$ such that asymptotic stability is achieved as in~\cite{McGovern99}.

In the offline phase, the \mbox{fixed-point} precision of the variables (the number of integer bits $l_i$ and the number of fractional bits $l_f$) is chosen such that there is no overflow and one of the conditions on $\epsilon$ is satisfied. Note that these conditions can be overly-conservative. This ensures that the MPC performance is guaranteed with a large margin, but the computation is slower because of the large number of bits.

\begin{remark}\label{rem:GM}
Instead of the FGM, we can use the simple gradient descent method, where fewer errors accumulate, and less memory is needed, but more iterations are necessary to reach to the optimal solution.
\end{remark}

Once $l_i$ and $l_f$ have been chosen, we can pick $N_\sigma$ such that the there is no overflow in Protocol~\ref{pro:FGM_CS} and Protocol~\ref{pro:FGM_SS}, respectively. 
In Protocol~\ref{pro:FGM_CS}, truncation cannot be done on encrypted data by the server, thus, the multiplications by $2^{l_f}$ are accumulated between $z_k$ and $t_k$, which means that $t_k$ is multiplied by $2^{3l_f}$. Hence, we choose $N_\sigma$ such that $l_i+3l_f + 2 < \log_2 (N_\sigma/3)$ holds. 
For Protocol~\ref{pro:FGM_SS} we pick $N_\sigma$ such that $l_i+3l_f + 3 + \lambda_{\sigma} < \log_2 (N_\sigma/3)$ holds, where $\lambda_{\sigma}$ is the statistical security parameter. 
If the $N_\sigma$ that satisfies these requirements is too large, one should use the simple gradient method, which will have multiplications up to $2^{2l_f}$. 

%%%%%%%%%%%%%%%%%%%%%%%%%%%%%%%%%%%%%%%%%%%%%%%%%%%%%%%%%%%%%%%%%%%%%%%%%%%%%%%%

\section{Numerical results and trade-off}~\label{sec:numerical}
The number of bits in the representation is crucial for the performance of the MPC scheme. At the same time, a more accurate representation slows down the private MPC computation. 

In Figure~\ref{fig:comparison}, for a toy model of a spacecraft from~\cite{MPCbenchmarks} and based on~\cite{Hegrenaes05}, we compare the predicted theoretical error bounds from Theorems~\ref{thm:quantization} and \ref{thm:round-off} and the norm of the actual errors between the control input obtained with the \mbox{client-server} protocol and the control input of the unencrypted MPC. 
The simulation is run for a time step of MPC with 18 iterations, with three choices of the number of fractional bits: 16, 24 and 32 bits. The results are similar for the \mbox{two-server} protocol. 
We can observe that the predicted errors are around two orders of magnitude larger than the real errors caused by the encryption, and even for 16 bits of precision, the actual error is small.

\begin{figure}[h]
  \centering
    \includegraphics[width=0.45\textwidth]{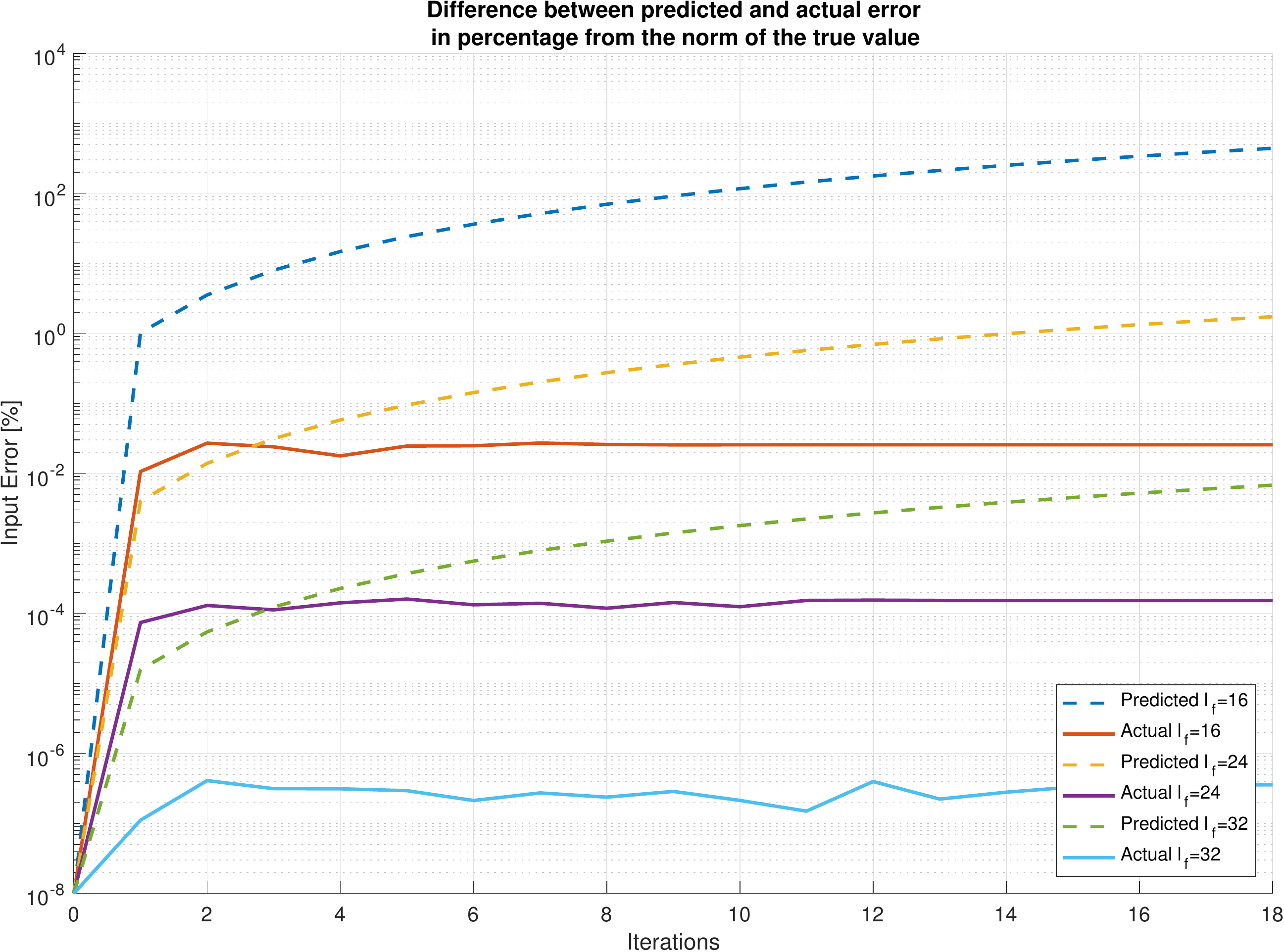}
   \vspace{-0.2cm}
  \caption{Predicted and actual errors for the control input computed in Protocol~\ref{pro:FGM_CS} for the Problem~\eqref{eq:mpc(i)}. The y-axis shows the errors (predicted and actual) as a percentage of the norm of the true iterate $U_k$.}   
   \label{fig:comparison}
\end{figure}

Furthermore, we implemented the above two protocols for various problem sizes. Table~\ref{tab:results} shows the computation times for Protocols~\ref{pro:FGM_CS} and~\ref{pro:FGM_SS} for $l_i = 16$ integer bits and $l_f\in\{16,32\}$ fractional bits. We fix the number of iterations to 50. 
The times obtained vary linearly with the number of iterations, so one can approximate how long it would take to run a different number of iterations. The size of the $N_\sigma$ for the encryption is chosen to be $512$ bits. The computations have been effectuated on a 2.2 GHz Intel Core~i7.

\begin{table}[h]
\small
\begin{center}
 \begin{tabular}{ c | c | c | c | c | c } 
 \hline
$l_f$ & $\begin{matrix}n = 2\\m=2\end{matrix}$ & $\begin{matrix}n = 5\\m=5\end{matrix}$ & $\begin{matrix}n = 10\\m=10\end{matrix}$ & $\begin{matrix}n = 20\\m=20\end{matrix}$ &$\begin{matrix}n = 50\\m=30\end{matrix}$ \\ [0.5ex] 
 \hline\hline
CS 16 & 1.21 & 4.08 & 13.38 & 39.56 & 84.28\\ 
 \hline
CS 32 &  1.33 & 5.20 & 15.46 & 53.48 & 105.84 \\
 \hline
SS 16 &  23.27 & 59.81 & 123.19 & 261.75 & 457.87\\ 
 \hline
SS  32 & 31.21 & 91.74 & 170.62 & 372.42 & 579.38 \\
 \hline
\end{tabular}
\caption{Average execution time in seconds for Protocols~\ref{pro:FGM_CS},~\ref{pro:FGM_SS}.}
\vspace{-0.4cm}
\label{tab:results}
\end{center}
\vspace{-0.4cm}
\end{table}

The larger the number of fractional bits, the larger is the execution time, because the computations involve larger numbers and more bits have to be sent from one party to another in the case of encrypted comparisons. 
Communication is the reason for the significant \mbox{slow-down} observed in the \mbox{two-server} architecture compared to the \mbox{client-server} architecture. 
Performing the projection with PHE requires $l$ communication rounds, where $l$ is the number of bits of the messages compared. 
Privately updating the iterates requires another communication round. 
However, the assumption is that the servers are powerful computers, and the execution time can be greatly improved using parallelization and more efficient data structures.

To summarize, privacy comes at the price of complexity: hiding the private data requires working with large encrypted numbers, and the computations on private data require communication. 
Furthermore, the more parties that need to be oblivious to the private data (two servers in the second architecture compared to one server in the first architecture), the more complex the private protocols become.

%%%%%%%%%%%%%%%%%%%%%%%%%%%%%%%%%%%%%%%%%%%%%%%%%%%%%%%%%%%%%%%%%%%%%%%%%%%%%%%%

\section{Conclusion}
In this paper, we have presented two methods to achieve the private computation of the solution to cloud-outsourced implicit MPC, through additively homomorphic encryption. 
The client, which is the owner of a linear-time discrete system with input constraints, desires to keep the measurements, the constraints and the control inputs private. 
The cloud, which can be composed by one or two servers, knows the parameters of the system and has to relieve the computation, i.e., solving the optimization problems, from the client in a private manner. 
The first method proposes an architecture where the computation is split between the client and one server. The client encrypts the current state with a PHE scheme, and performs the nonlinear operations (the projection on the constraints), while the server performs the rest of the computations needed to solve the optimization problem via the fast gradient method. 
Secondly, we propose a \mbox{two-server} architecture, in which the client is exempt from any computations apart from encrypting the state and decrypting the input. The two servers engage in a protocol to solve the optimization problem and use communication to perform the nonlinear operations on encrypted data. We prove that both protocols are private for \mbox{semi-honest} servers.

In order to use encryptions, we map the real values to rational \mbox{fixed-point} values. We analyze the errors introduced by the associated quantization and \mbox{round-off} errors and give upper bounds which can be used to choose an appropriate precision that corresponds to performance requirements on the MPC. Furthermore, we implement the two proposed protocols and perform numerical tests.

The two architectures proposed in this paper offer insights for a more common architecture met in practice of multiple clients and one server, which we will consider in future work.

%%%%%%%%%%%%%%%%%%%%%%%%%%%%%%%%%%%%%%%%%%%%%%%%%%%%%%%%%%%%%%%%%%%%%%%%%%%%%%%%

\section{Acknowledgements}
This work was supported in part by ONR N00014-17-1-2012, and by NSF CNS-1505799 grant and the Intel-NSF Partnership for Cyber-Physical Systems Security and Privacy.

%%%%%%%%%%%%%%%%%%%%%%%%%%%%%%%%%%%%%%%%%%%%%%%%%%%%%%%%%%%%%%%%%%%%%%%%%%%%%%%%

\small

\bibliographystyle{IEEEtran}
\bibliography{IEEEabrv,biblo}

\end{document}